\documentclass[10pt, a4paper, oneside, reqno]{amsart}
\usepackage{graphicx}
\usepackage{float}
\usepackage{pstricks}
\usepackage{amscd}
\usepackage{amsmath}
\usepackage{amsxtra}
\usepackage{hyperref}
\usepackage[T1]{fontenc}
\usepackage[utf8]{inputenc}
\usepackage{lipsum}
\usepackage{tikz}
\usepackage{amssymb, amsthm}
\usepackage{url}
\usepackage{enumerate}
\usepackage{mathrsfs}
\usepackage{epsfig}
\usepackage{soul}
\usepackage[active]{srcltx}

\usetikzlibrary{arrows}
\usetikzlibrary{calc}
\theoremstyle{plain}
\newtheorem{theorem}{Theorem}[section]
\newtheorem{c-theorem}{Construction theorem}[section]

\newtheorem{lemma}[theorem]{Lemma}
\newtheorem{proposition}[theorem]{Proposition}
\newtheorem{corollary}[theorem]{Corollary}

\theoremstyle{definition}

\newtheorem{question}[theorem]{Question}

\theoremstyle{remark}
\newtheorem{remark}[theorem]{Remark}

\makeatletter
\@namedef{subjclassname@2020}{%
\textup{2020} Mathematics Subject Classification}
\makeatother

\newcommand{\ncom}{\newcommand}
\ncom{\bq}{\begin{equation}}
\ncom{\eq}{\end{equation}}
\ncom{\beqn}{\begin{eqnarray*}}
\ncom{\eeqn}{\end{eqnarray*}}
\ncom{\beq}{\begin{eqnarray}}
\ncom{\eeq}{\end{eqnarray}}
\ncom{\nno}{\nonumber}
\ncom{\rar}{\rightarrow}
\ncom{\Rar}{\Rightarrow}
\ncom{\noin}{\noindent}
\ncom{\bc}{\begin{centre}}
\ncom{\ec}{\end{centre}}
\ncom{\sz}{\scriptsize}
\ncom{\rf}{\ref}
\ncom{\sgm}{\sigma}
\ncom{\Sgm}{\Sigma}
\ncom{\dt}{\delta}
\ncom{\Dt}{Delta}
\ncom{\lmd}{\lambda}
\ncom{\Lmd}{\Lambda}
\ncom{\eps}{\epsilon}
\ncom{\pcc}{\stackrel{P}{>}}
\ncom{\dist}{{\rm\,dist}}
\ncom{\im}{{\rm Im\,}}
\ncom{\sgn}{{\rm sgn\,}}
\ncom{\ba}{\begin{array}}
\ncom{\ea}{\end{array}}
\ncom{\eop}{\hfill{{\rule{2.5mm}{2.5mm}}}}
\ncom{\eof}{\hfill{{\rule{1.5mm}{1.5mm}}}}
\ncom{\hone}{\mbox{\hspace{1em}}}
\ncom{\htwo}{\mbox{\hspace{2em}}}
\ncom{\hthree}{\mbox{\hspace{3em}}}
\ncom{\hfour}{\mbox{\hspace{4em}}}
\ncom{\hsev}{\mbox{\hspace{7em}}}
\ncom{\vone}{\vskip 2ex}
\ncom{\vtwo}{\vskip 4ex}
\ncom{\vonee}{\vskip 1.5ex}
\ncom{\vthree}{\vskip 6ex}
\ncom{\vfour}{\vspace*{8ex}}
\ncom{\norm}{\|\;\;\|}
\ncom{\integ}[4]{\int_{#1}^{#2}\,{#3}\,d{#4}}
\ncom{\inp}[2]{\langle{#1},\,{#2} \rangle}
\ncom{\Inp}[2]{\Langle{#1},\,{#2} \Langle}
\ncom{\vspan}[1]{{{\rm\,span}\#1 \}}}
\ncom{\dm}[1]{\displaystyle {#1}}

\begin{document}
\title[L\'evy measures for Dirichlet-type spaces]{L\'evy measures for Dirichlet-type spaces on the unit bidisc}
\author[Santu Bera and Shanola S. Sequeira]{Santu Bera and Shanola S. Sequeira}

\address{Department of Mathematics \\
Indian Institute of Technology  Kharagpur, India}
\email{santu@maths.iitkgp.ac.in}

\address{School of Mathematics and Computer Science \\
Indian Institute of Technology Goa, India}
\email{shanolasequeira@gmail.com, shanola.crd208@iitgoa.ac.in}

	
	

\begin{abstract}
        For the Dirichlet-type space $\mathcal D(\rho^{(1)},\rho^{(2)})$ on the unit bidisc $\mathbb D^2,$ where $\rho^{(1)}$ and $\rho^{(2)}$ are finite positive Borel measures on the closed unit disc $\overline{\mathbb D},$ we show that the L\'evy measure $\nu_{(\mathscr M_z,1)}$ associated with the completely alternating multisequence $\left\{\|z^\alpha\|_{\mathcal D(\rho^{(1)},\rho^{(2)})}^2
\right\}_{\alpha\in\mathbb Z_+^2}$ admits the explicit representation
        \begin{equation*}
		d\nu_{(\mathscr M_z,1)}(x)=\frac{1}{1-x_1}d((S_{*}\rho^{(1)})\times\delta_1)(x)+\frac{1}{1-x_2}d(\delta_1\times (S_{*}\rho^{(2)}))(x) 
		\end{equation*}	
		on $[0,1]^2\backslash\{(1,1)\},$ where $S_{*}\rho^{(i)},$ $i=1,2$, denotes the pushforward measure  of $\rho^{(i)}$ by the map  $S:\overline{\mathbb D}\rightarrow [0,1]$ defined by $S(z)=|z|^2,$ $z\in \overline{\mathbb D},$ and $\delta_1$ denotes the Dirac measure at 1.
	\end{abstract}
	\maketitle
	\section{Introduction and Preliminaries}
 
    The theory of Dirichlet-type spaces has been a subject of extensive research over the past several decades, owing to its rich function-theoretic structure and its deep connections with operator theory. A fundamental result of Richter \cite{R1991} shows that harmonically weighted Dirichlet-type spaces on the unit disc serve as model spaces for cyclic analytic $2$-isometries. Aleman \cite[Chapter~IV]{A1993} subsequently extended this result to the broader class of completely hyperexpansive operators by considering superharmonically weighted Dirichlet-type spaces on the unit disc. Recently, several-variable counterparts of these spaces and their associated operator models have been developed on the unit ball \cite{CGR2020} and on the unit bidisc \cite{BCG2025,B2025}. The present paper addresses the measure-theoretic structure underlying the bidisc model for completely hyperexpansive operator pairs. More precisely, we relate the L\'evy measure arising from the L\'evy--Khinchin representation of the multiplication pair to the measures used to define the corresponding Dirichlet-type space. 

Throughout the paper, $\mathbb D:= \{z \in \mathbb C : |z| < 1\} $ denotes the open unit disc  in the complex plane $\mathbb C$ and $\mathbb Z_+$ denotes the set of nonnegative integers. For a positive integer $d,$ let $\mathbb Z_+^d$ denote the $d$-fold Cartesian product of $\mathbb Z_+,$ which forms a semigroup under coordinatewise addition. Given $\alpha=(\alpha_1,\ldots,\alpha_d)$ and $\beta=(\beta_1,\ldots,\beta_d)$ in $\mathbb Z^d_+$, we write $|\beta|:=\sum_{i=1}^d\beta_i,$ $\binom{\alpha}{\beta}:=\prod_{i=1}^d\binom{\alpha_i}{\beta_i}$ and write $\beta\leq \alpha$ when $\beta_i\leq \alpha_i$ for every $i\in\{1,\ldots,d\}.$ For a map $\phi:\mathbb Z^d_+\rightarrow \mathbb R$, the $i$-th ($1\leq i\leq d$) backward difference operator (respectively, the forward difference operator) is defined by $\nabla_i\phi(\alpha)=\phi(\alpha)-\phi(\alpha+\varepsilon(i))$ (respectively, $\Delta_i\phi(\alpha)=\phi(\alpha+\varepsilon(i))-\phi(\alpha)$), where $\alpha\in \mathbb Z_+^d$ and $\varepsilon(i)$ is the $d$-tuple whose $i$-th coordinate is $1$ and remaining coordinates are $0$. Note that the operators $\nabla_1,\ldots,\nabla_d$ commute and hence for $\alpha\in\mathbb Z_+^d$, we set
$\nabla^\alpha:=\nabla_1^{\alpha_1}\cdots\nabla_d^{\alpha_d}.$ Following \cite{BCR1984}, a map $\phi: \mathbb{Z}_+^d \to \mathbb{R}$ is called  \textit{completely monotone} if $(\nabla^{\alpha}\phi)(\beta) \geq 0$ for each $\alpha,\beta \in \mathbb{Z}_+^d$ and $\phi$ is called \textit{completely alternating} if $(\nabla^{\alpha}\phi)(\beta) \leq 0$ for each $\beta \in \mathbb{Z}_+^d$ and $\alpha \in \mathbb{Z}_+^d\backslash \{0\}$.
	
	We recall the well-known measure-theoretic characterizations of these two classes (see \cite[Ch.~4, Propositions~6.11~\&~6.12]{BCR1984}):
	\begin{proposition}\label{comp-mono-alter}
		Let $\phi: \mathbb{Z}_+^d \to \mathbb{R}$. Then the following holds:
		\begin{itemize}
			\item[$\mathrm{(i)}$]  $\phi$ is completely monotone if and only if  there exists a finite positive Radon measure $\mu$ on $[0,1]^d$ such that 
			\begin{equation}\label{hausdorff-moment}
			\phi(\alpha) =  \int_{[0,1]^d} x^\alpha \, d\mu(x), \quad  \alpha \in \mathbb{Z}^d_+.
			\end{equation}
			\item[$\mathrm{(ii)}$] $\phi$ is completely alternating if and only if there exist $a \in \mathbb{R}$, an additive map $b: \mathbb{Z}_+^d \to \mathbb{R}_+$, and a positive Radon measure $\nu$ on $[0,1]^d \backslash \{\mathbf{1}\}$ $($where $\mathbf{1} = (1,\ldots,1))$ such that 
			\begin{equation}\label{LK-repn}
			\phi(\alpha) = a + b(\alpha) + \int_{[0,1]^d \backslash \{\mathbf{1}\}} (1-x^\alpha) \, d\nu(x), \quad \alpha \in \mathbb{Z}^d_+ \backslash \{0\},
			\end{equation}
			where 
			\beq\label{b-additive}
	 a = \phi(0) \text{ and } \quad b(\alpha) = \lim_{n \to \infty} \frac{\phi(n\alpha)}{n}, \quad \alpha \in \mathbb{Z}_+^d.
	 \eeq
		\end{itemize}
	\end{proposition}

	\begin{remark}
   Hausdorff \cite{H1923} obtained the one-variable characterization in Proposition~\ref{comp-mono-alter}(i) in terms of Hausdorff moment sequences, and its multivariable counterpart appeared in \cite{HS1933}.  The unique measure $\mu$ in \eqref{hausdorff-moment} is called the \emph{representing measure} for the completely monotone function $\phi$. The representation \eqref{LK-repn} is a special case of the L\'evy--Khinchin representation on abelian semigroups. 
        The triple $(a,b,\nu),$ appearing in \eqref{LK-repn}, referred to as the \textit{L\'evy triple} for the completely alternating function $\phi,$ is uniquely determined by $\phi$ (see \cite[Theorems~3.7~\&~4.4]{BCR1976}). The measure $\nu$ in \eqref{LK-repn} is called the {\it L\'evy measure} for $\phi.$ 
        For further background on the Hausdorff moment problem and completely alternating functions, see \cite{BCR1976,BCR1984,BD2005,RZ2019}.\end{remark}
		
	These representation theorems have natural operator-theoretic interpretations.
	Proposition~\ref{comp-mono-alter}(i) underlies Athavale's multivariable extension \cite{At1987} of Agler's criterion \cite{A1985} for contractive subnormal operators. The importance of Proposition~\ref{comp-mono-alter}(ii) is discussed in \cite{At1996,AS1999}, where it is linked to completely hyperexpansive operators. 
	To make this connection precise, consider a commuting $d$-tuple $\mathbf T=(T_1,\ldots,T_d)$ on a complex Hilbert space $\mathcal H,$ that is, $T_1,\ldots,T_d$ are bounded linear operators on $\mathcal H$  such that $T_iT_j=T_jT_i$ for all $1\leq i,j\leq d.$ Following \cite{A1990}, for each $\alpha\in \mathbb Z_+^d,$ we define
	\begin{equation}\label{CH}
	B_\alpha({\bf T}):=\sum_{\substack{\beta\in \mathbb Z_+^d \\ 0\leq \beta \leq \alpha}}(-1)^{|\beta|}\binom{\alpha}{\beta}{\bf T}^{*\beta}{\bf T}^{\beta},
	\end{equation}
	where  ${\bf T}^{\beta}=T^{\beta_1}_1\cdots T^{\beta_d}_d.$ Recall that a commuting $d$-tuple ${\bf T}$ is called {\it completely hypercontractive} if $B_\alpha({\bf T})\geq 0$ for all $\alpha\in \mathbb Z_+^d$ (see \cite{A1985, At1987}).  On the other hand, a commuting $d$-tuple ${\bf T}$ is called {\it completely hyperexpansive} if $B_\alpha({\bf T})\leq 0$ for all $\alpha\in \mathbb Z_+^d \backslash \{0\}$ (see \cite{At1996, AS1999}). By \cite[Remark~2]{AS1999}, ${\bf T}$ is completely hyperexpansive if and only if for each $h\in \mathcal H$, the function $\phi_{({\bf T},h)}: \mathbb{Z}_+^d \to \mathbb{R}$ defined by
	$$ \phi_{({\bf T},h)}(\alpha) := \|\mathbf{T}^\alpha h\|^2, \quad \alpha \in \mathbb{Z}_+^d $$ 
	is completely alternating.
	Consequently, by Proposition \ref{comp-mono-alter}(ii), there exist an additive map $b_{({\bf T},h)}: \mathbb{Z}_+^d \to \mathbb{R}_+$ and a positive Radon measure $\nu_{({\bf T},h)}$ on $[0,1]^d \backslash \{\mathbf{1}\}$ such that for all $\alpha \in \mathbb{Z}^d_+ \backslash \{0\}$,
	\begin{equation}\label{eq_multivarible}
	\phi_{({\bf T},h)}(\alpha)=\|{\bf T}^\alpha h\|^2=\|h\|^2+b_{({\bf T},h)}(\alpha) +\int_{[0,1]^d\backslash \{\mathbf{1}\}}(1-x^\alpha)d\nu_{({\bf T},h)}(x).
	\end{equation}
	We refer to \eqref{eq_multivarible} as the \textit{L\'evy-Khinchin representation} for the tuple $(\mathbf{T}, h)$ and to $\nu_{({\bf T},h)}$ as the \textit{L\'evy measure} for $(\mathbf{T}, h)$. 
	By definition, each $T_i$ is completely hyperexpansive for $i\in \{1,\ldots,d\},$ and hence for each $h\in \mathcal H,$ there exist an additive map $b_{(T_i,h)}: \mathbb{Z}_+ \to \mathbb{R}_+$ and a positive Radon measure $\nu_{(T_i,h)}$ on $[0,1)$ such that 
	\begin{equation}\label{eq_onevarible}
	\|T_i^n h\|^2=\|h\|^2+b_{(T_i,h)}(n) +\int_{[0,1)}(1-x^n)d\nu_{(T_i,h)}(x),\quad n\geq 0.
	\end{equation}
	
	In \cite{B2025}, an analytic model for completely hyperexpansive operator pair is developed in terms of Dirichlet-type spaces on the unit bidisc as follows: let $H^2(\mathbb D^2)$ denote Hardy space on the unit bidisc $\mathbb D^2$. Given finite positive Borel measures $\rho^{(1)}$ and $\rho^{(2)}$ on the closed unit disc $\overline{\mathbb D}$, the Dirichlet-type space $\mathcal D(\rho^{(1)},\rho^{(2)})$ (see \cite[Definition~1.1]{B2025}) is defined by
	\[{\mathcal D}(\rho^{(1)},\rho^{(2)}) = \{f \in H^2(\mathbb D^2): D_{\rho^{(1)},\rho^{(2)}}(f) <\infty\},\]
	where 
	\begin{align*} 
	D_{\rho^{(1)},\rho^{(2)}}(f) &:=\sup_{0 < r < 1}\int_{0}^{2\pi} \int_{\mathbb D}|\partial_1 f(z_1, re^{i \theta})|^2 U_{\rho^{(1)}}(z_1) \,dA(z_1)\frac{d\theta}{2\pi} \\
	&+ \sup_{0 < r < 1} \int_{0}^{2\pi} \int_{\mathbb D}|\partial_2 f(re^{i \theta}, z_2)|^2 U_{\rho^{(2)}}(z_2) \,dA(z_2)\frac{d\theta}{2\pi},
	\end{align*}
	and the norm  is \begin{equation*}
	\|f\|^2_{\mathcal D({\rho^{(1)},\rho^{(2)}})}:=\|f\|^2_{H^2(\mathbb D^2)}+D_{\rho^{(1)},\rho^{(2)}}(f), \ \ f \in \mathcal D({\rho^{(1)},\rho^{(2)}}).
	\end{equation*}
	According to \cite[Theorem~1.2]{B2025}, a cyclic analytic completely hyperexpansive pair ${\bf T}=(T_1,T_2)$ with vanishing defect, where the defect operator is given by
	\begin{equation*}
	D(T_1,T_2) := I-T_1^*T_1-T_2^*T_2+T_1^*T_2^*T_1T_2,
	\end{equation*}
	is unitarily equivalent to  $\mathscr M_z=(\mathscr M_{z_1},\mathscr M_{z_2})$ on $\mathcal D(\rho^{(1)},\rho^{(2)})$ for some $\rho^{(1)}$ and $\rho^{(2)}$ on $\overline{\mathbb D}$ if and only if the joint kernel $\ker {\bf T}^*=\ker T_1^*\cap \ker T_2^*$ is a one-dimensional generating wandering subspace. 
    
    Since $\mathscr M_z$ is completely hyperexpansive, every $h\in\mathcal D(\rho^{(1)},\rho^{(2)})$ gives rise to a L\'evy measure $\nu_{(\mathscr M_z,h)}$. This leads to the following natural question.
    \begin{question} \label{main-ques}
	For the Dirichlet-type space $\mathcal D(\rho^{(1)},\rho^{(2)})$, how is the L\'evy measure $\nu_{(\mathscr M_z,h)}$ for $(\mathscr M_z,h)$ related to the defining measures $\rho^{(1)}$ and $\rho^{(2)}$?\end{question}
Interestingly, for $h=1,$ the L\'evy measure for $(\mathscr M_{z},1)$ on $\mathcal D(\rho^{(1)},\rho^{(2)})$ decomposes in terms of $\rho^{(1)}$ and $\rho^{(2)}.$ To state the main result, we first recall the notion of {\it pushforward measure} (cf. \cite[Ch. 2, 1.14]{BCR1984}). Let $(X_1,\Sigma_1)$ and $(X_2,\Sigma_2)$ be two measure spaces and  $S:X_1\rightarrow X_2$ be a measurable map. Let $\nu$ be a measure on  $(X_1,\Sigma_1)$. Then the pushforward of $\nu$ by $S$ is defined as 
	\begin{equation*}
	S_{*}\nu(E):=\nu(S^{-1}(E)), \quad E\in \Sigma_2.
	\end{equation*}
	In this case, for any measurable map $g$ on $X_2$, we have
	\begin{equation}\label{push-forward-int}
	\int_{X_2}g(y)\, d(S_{*}\nu)(y)=\int_{X_1}(g\circ S)(x) d\nu(x).
	\end{equation}
	\begin{theorem}\label{d-mu-1-2-decomp}
		Let $\rho^{(1)},\rho^{(2)}$ be finite positive Borel measures on the closed unit disc $\overline{\mathbb D}.$ Then the L\'evy measure $\nu_{(\mathscr M_z,1)}$ for the tuple $(\mathscr M_{z},1)$  on $\mathcal D(\rho^{(1)},\rho^{(2)})$ is of the form
		\begin{equation}\label{levy-k-diri}
		d\nu_{(\mathscr M_z,1)}(x)=\frac{1}{1-x_1}d((S_{*}\rho^{(1)})\times\delta_1)(x)+\frac{1}{1-x_2}d(\delta_1\times (S_{*}\rho^{(2)}))(x) 
		\end{equation}	
		on $[0,1]^2\backslash\{(1,1)\},$ where $S:\overline{\mathbb D}\rightarrow [0,1]$ defined by $S(z)=|z|^2,$ $z\in \overline{\mathbb D}.$
	\end{theorem}


The article is organized as follows: In Section~\ref{section-2}, we use the uniqueness of L\'evy triples to characterize completely alternating functions $\phi:\mathbb Z_+^d\to\mathbb R$ satisfying
$\phi(\alpha)=\sum_{i=1}^d \phi^{(i)}(\alpha_i)-(d-1)\phi(0), \; \alpha=(\alpha_1,\ldots,\alpha_d)\in \mathbb Z^d_+$
	  (see Theorem~\ref{main-theorem}). As an application, we derive an operator-theoretic analog of Theorem~\ref{main-theorem} for completely hyperexpansive $d$-tuples (see Theorem~\ref{main-thm-new}). In Section~\ref{dirichlet-space}, we present a proof of  Theorem~\ref{d-mu-1-2-decomp}. It is worth noting that this proof relies on a one-variable counterpart of Theorem~\ref{d-mu-1-2-decomp}, which appears to be unnoticed in the literature.
       

	\section{Completely hyperexpansive tuples with zero defect}\label{section-2} 
	Let $\phi: \mathbb Z^d_+\to \mathbb R.$ It follows from \cite[Ch.~4, Lemma~6.3]{BCR1984} that $\phi$ is completely alternating if and only if each forward difference $\Delta_i \phi$ is completely monotone for $i=1,2,\ldots,d.$ For a completely alternating function $\phi$ on $\mathbb Z^d_+,$ let $\mu_i$ be the representing measures for $\Delta_i \phi, \; i=1,2\ldots, d,$  as given in \eqref{hausdorff-moment}. Then by \cite[Theorem~1]{A2003} (cf. \cite[Remark~1]{At1996}), we have
	\beq
	\label{mu-i-nu} d\mu_i(x)&=&(1-x_i)d\nu(x) \text{ on } [0,1]^d\backslash \{\boldsymbol{1}\} \text{ for all } i=1,2,\ldots,d,\\	 
	 \label{b-alpha} b(\alpha)&=&\sum_{i=1}^d \alpha_i \mu_i(\{\boldsymbol{1}\})\text{ for all } \alpha=(\alpha_1,\ldots,\alpha_d)\in \mathbb Z_+^d.\eeq  

	Let $\phi$ be a completely alternating function on $\mathbb Z_+^d.$ For
$i=1,\ldots,d$, define its $i$-th coordinate function by
\begin{equation}\label{phi-i-defn}
 \phi^{(i)}(n):=\phi(n\varepsilon(i)),\qquad n\in\mathbb Z_+.
\end{equation}
Then $\phi^{(i)}$ is completely alternating on $\mathbb Z_+$. So by Proposition~\ref{comp-mono-alter}(ii), there exist an additive map $b^{(i)}:\mathbb Z_+\rightarrow \mathbb R_+$ and a positive Radon measure $\nu^{(i)}$ on $[0,1)$ such that 
	\begin{equation}\label{levy-kinchin-one-v}
	\phi^{(i)}(n) = \phi(0) + b^{(i)}(n) + \int_{[0,1)} (1-x^n) \, d\nu^{(i)}(x), \quad  n\geq 0.
	\end{equation}
	Thus $(\phi(0),b^{(i)},\nu^{(i)})$ is the L\'evy triple for $\phi^{(i)},$ $i=1,\ldots,d.$  The function $\phi^{(i)}$ is referred to as the {\it $i$-th coordinate function} of $\phi$ and the corresponding L\'evy measure $\nu^{(i)}$ is called the {\it $i$-th coordinate L\'evy measure} for $\phi.$ Moreover, $\Delta \phi^{(i)}(n)= \phi^{(i)}(n+1)-\phi^{(i)}(n),\; n\in \mathbb Z_+$ is completely monotone. Let $\mu^{(i)}$ denotes the representing measure (as presented in \eqref{hausdorff-moment}) for $\Delta\phi^{(i)}(n).$ Applying the one-variable versions of \eqref{mu-i-nu} and
\eqref{b-alpha}, the equation \eqref{levy-kinchin-one-v} becomes
	\beq\label{one-variable-levy}
	\phi^{(i)}(n) = \phi(0) + n\mu^{(i)}(\{1\}) + \int_{[0,1)} (1-x^n) \, \frac{d\mu^{(i)}(x)}{1-x}, \quad  n\in \mathbb Z_+.
	\eeq 
	
	The following lemma plays a crucial role in the proof of the main result of this section.
	\begin{lemma}\label{bh sum}
		Let $\phi$ be a completely alternating function on $\mathbb Z_+^d$ with L\'evy triple $(\phi(0),b,\nu)$ and let $(\phi(0),b^{(i)},\nu^{(i)})$ be the L\'evy triple for the $i$-th coordinate function $\phi^{(i)}$, $i=1,\ldots,d.$ Then 
		\begin{equation*}
		b(\alpha)= b^{(1)}(\alpha_1)+\ldots+b^{(d)}(\alpha_d),\quad \alpha=(\alpha_1,\ldots,\alpha_d)\in \mathbb Z_+^d.
		\end{equation*}
	\end{lemma} 
	\begin{proof}
		Since $b$ is additive, for any $\alpha = (\alpha_1,\ldots,\alpha_d)\in \mathbb Z_+^d$,
		\begin{align*}
		b(\alpha) &= b(\alpha_1\varepsilon(1))+\ldots+b(\alpha_d\varepsilon(d))\overset{\eqref{b-additive}}{=}\sum_{i=1}^d\lim_{k\to \infty}\frac{\phi(k\alpha_i\varepsilon(i))}{k}\\
		&\overset{\eqref{phi-i-defn}}{=} \sum_{i=1}^d\lim_{k\to \infty}\frac{\phi^{(i)}(k\alpha_i)}{k}\overset{\eqref{b-additive}}{=} \sum_{i=1}^d b^{(i)}(\alpha_i).
		\end{align*} 
		This completes the proof.
	\end{proof} 
    
	The next result precisely characterizes when the L\'evy measure of a completely alternating function admits a decomposition determined solely by its coordinate L\'evy measures.
	
	\begin{theorem}\label{main-theorem}
		Let $\phi$ be a completely alternating function on $\mathbb Z^d_+$ with L\'evy measure $\nu.$ Let $\phi^{(i)}$ denote the $i$-th coordinate function with L\'evy measure $\nu^{(i)},$ $i=1,\ldots,d.$ Then the following statements are equivalent:
		\begin{itemize}
			\item[$(\mathrm{i})$] for each $\alpha=(\alpha_1,\ldots,\alpha_d)\in \mathbb Z^d_+,$ $\phi$ satisfies 
			\beqn
			\phi(\alpha)=\sum_{i=1}^d \phi^{(i)}(\alpha_i)-(d-1)\phi(0), 
			\eeqn 
			\item[$(\mathrm{ii})$]  the L\'evy measure $\nu$ for $\phi$ is of the form
			\beq\label{nu-sum}
	         \nu=\sum_{i=1}^d \delta_1\times \cdots \times \underbrace{\nu^{(i)}}_{\text{{$i$-th}}} \times \cdots \times \delta_1.
	        \eeq
		\end{itemize}
	\end{theorem}
	\begin{proof}
		$(\mathrm{i})\Rightarrow (\mathrm{ii}):$ Let $\phi$ be a completely alternating function on $\mathbb Z_+^d$ satisfying
		\beqn
			\phi(\alpha)=\sum_{i=1}^d \phi^{(i)}(\alpha_i)-(d-1)\phi(0),\quad \alpha\in \mathbb Z^d_+.
		\eeqn 
		Since each $\phi^{(i)}$ ($i=1,\ldots,d$) is completely alternating, an application of Lemma~\ref{bh sum} together with \eqref{LK-repn} and \eqref{levy-kinchin-one-v} yields
		\beq\label{sum-dnu-i}
		\int_{[0,1]^d \backslash \{\mathbf{1}\}} (1-x^\alpha)d\nu(x)= \sum_{i=1}^d \int_{[0,1)} (1-t^{\alpha_i}) d\nu^{(i)}(t).
		\eeq
		Define $R_i := [0,1]\times\cdots\times\underbrace{[0,1)}_{\text{{$i$-th}}}\times \cdots \times[0,1]$ for $i=1,\ldots,d.$ Then \eqref{sum-dnu-i} becomes
		\beqn
		&&\int_{[0,1]^d \backslash \{\mathbf{1}\}} (1-x^\alpha)d\nu(x)\\&=& \sum_{i=1}^d \int_{R_i} (1-x^{\alpha}) d(\delta_1\times \cdots \times \underbrace{\nu^{(i)}}_{\text{{$i$-th}}} \times \cdots \times \delta_1)(x)\\
		&=& \int_{[0,1]^d \backslash \{\mathbf{1}\}} (1-x^{\alpha}) d\left(\sum_{i=1}^d \delta_1\times \cdots \times \underbrace{\nu^{(i)}}_{\text{{$i$-th}}} \times \cdots \times \delta_1\right)(x).
		\eeqn
		Thus, the uniqueness of the L\'evy triple $(\phi(0),b,\nu)$ of $\phi$ yields \eqref{nu-sum}. 
		
		$(\mathrm{ii})\Rightarrow (\mathrm{i}):$ Let $(\phi(0),b,\nu)$ be the L\'evy triple for $\phi$, i.e., 
		\begin{equation*}
		\phi(\alpha) = \phi(0) + b(\alpha) + \int_{[0,1]^d \backslash \{\mathbf{1}\}} (1-x^\alpha) \, d\nu(x), \quad \alpha \in \mathbb{Z}^d_+ \backslash \{0\}.
		\end{equation*} 
	 Substituting \eqref{nu-sum} in the above equation and using Lemma~\ref{bh sum} gives
		\beqn
		\phi(\alpha)&=& \phi(0)+ b(\alpha) + \int_{[0,1]^d \backslash \{\mathbf{1}\}} (1-x^\alpha) \,  d\left(\sum_{i=1}^d \delta_1\times \cdots \times \underbrace{\nu^{(i)}}_{\text{{$i$-th}}} \times \cdots \times \delta_1\right)(x)\\
		&=&\phi(0)+ \sum_{i=1}^d b^{(i)}(\alpha_i)+\sum_{i=1}^d  \int_{[0,1)} (1-x_i^{\alpha_i}) d\nu^{(i)}(x_i)\\
		&=& \sum_{i=1}^d\phi^{(i)}(\alpha_i)-(d-1)\phi(0).
		\eeqn
		Hence the result.
	\end{proof}
    \begin{remark}
         The decomposition \eqref{nu-sum} is closely related to the following fact: if $(X_t, Y_t)$ is a non-Gaussian L\'evy process with L\'evy measure $\nu$, then $X_t$ and $Y_t$ are independent if and only if they do not jump simultaneously. In this case, $$\nu = \nu_X \times \delta_0+\delta_0 \times \nu_Y,$$ where $\nu_X$ and $\nu_Y$ are L\'evy measures of $X_t$ and $Y_t$, respectively; see \cite[Proposition 5.3]{CP2004} for further details.
    \end{remark}
	The result below shows that pairwise defect zero for a completely hyperexpansive operator $d$-tuple forces the associated L\'evy measure to have a decomposition similar to \eqref{nu-sum}.
	\begin{theorem}\label{main-thm-new}
		Let ${\bf T}=(T_1,\ldots,T_d)$  be a completely hyperexpansive operator $d$-tuple on $\mathcal H.$ Then  
        \beq\label{defect-i-j}		
		D(T_i,T_j)=I-T_i^*T_i-T_j^*T_j+T_i^*T_j^*T_iT_j=0,\quad 1\leq i\neq j \leq d,
		\eeq
		 if and only if for each $h\in \mathcal H,$ the L\'{e}vy measure $\nu_{({\bf T},h)}$ for $(\mathbf{T}, h)$ decomposes as 
		\begin{equation}\label{nu-prod}
		\nu_{(\mathbf T,h)}
=\sum_{i=1}^d
\delta_1\times\cdots\times
\underbrace{\nu_{(T_i,h)}}_{\text{$i$-th}}
\times\cdots\times\delta_1,
		\end{equation}
		where $\nu_{(T_i,h)}$ denotes the L\'evy measure for $(T_i,h).$ 
	\end{theorem}
    Before proving Theorem~\ref{main-thm-new}, let us prove the following lemma.
	\begin{lemma}\label{induction-lemma}
	Let ${\bf T}=(T_1,\ldots,T_d)$ be a commuting $d$-tuple on $\mathcal H.$ Then ${\bf T}$ satisfies \eqref{defect-i-j} if and only if 
	\beq\label{induction}
	{\bf T}^{*\alpha}{\bf T}^{\alpha}=\sum_{i=1}^dT^{*\alpha_i}_iT^{\alpha_i}_i-(d-1)I,\quad \alpha\in \mathbb Z^d_+.
	\eeq
	\end{lemma}
	\begin{proof} $(\Leftarrow)$ Follows by taking $\alpha=\varepsilon(i)+\varepsilon(j)$  $(1\leq i\neq j\leq d)$ in \eqref{induction}.
	
	$(\Rightarrow)$ Using induction we first show that for any $m\geq 0,$
	\beq\label{defect-i-j-m}
	T_i^*T^{*m}_jT_iT^m_j=T_i^*T_i+T^{*m}_jT^m_j-I, \; 1\leq i\neq j\leq d.
	\eeq 
	Clearly, \eqref{defect-i-j-m} is true for $m=0.$ Let us assume that \ref{defect-i-j-m} is true for $m.$ Now for $1\leq i\neq j\leq d,$
	\beqn
	T_i^*T^{*m+1}_jT_iT^{m+1}_j&=&T^*_j(T_i^*T^{*m}_jT_iT^{m}_j)T_j =T^*_j(T_i^*T_i+T^{*m}_jT^m_j-I)T_j\\
	&=& T_i^*T^*_jT_iT_j+T^{*m+1}_jT^{m+1}_j-T^*_jT_j\\
	&\overset{\eqref{defect-i-j}}{=}& T_i^*T_i+T^{*m+1}_jT^{m+1}_j-I
	\eeqn
	Thus, by induction \eqref{defect-i-j-m} is true for all $m\geq 0.$ To prove \eqref{induction}, we proceed by strong induction on $|\alpha|=\sum_{i=1}^d\alpha_i.$ Note that \eqref{induction} holds for $0\leq |\alpha| \leq 1.$ Let us assume that \eqref{induction} holds for $0\leq |\alpha| \leq n.$ Suppose that $|\alpha|=n+1.$ Choose $i\in \{1,\ldots,d\}$ such that $\alpha_i\geq 1.$ Then by the induction hypothesis 
	\beqn
	{\bf T}^{*\alpha}{\bf T}^{\alpha}&=&T_i^*({\bf T}^{*\alpha-\varepsilon(i)}{\bf T}^{\alpha-\varepsilon(i)})T_i\\
	&=&T^{*\alpha_i}_iT^{\alpha_i}_i+\sum_{\substack{j=1\\j\neq i}}^dT_i^*T^{*\alpha_j}_jT^{\alpha_j}_jT_i-(d-1)T_i^*T_i\\
	&\overset{\eqref{defect-i-j-m}}{=}&T^{*\alpha_i}_iT^{\alpha_i}_i+ \sum_{\substack{j=1\\j\neq i}}^d(T_i^*T_i+T^{*\alpha_j}_jT^{\alpha_j}_j-I)-(d-1)T_i^*T_i\\
	&=& \sum_{j=1}^d T^{*\alpha_j}_jT^{\alpha_j}_j-(d-1)I.
	\eeqn
	Thus by induction \eqref{induction} holds for each $\alpha\in \mathbb Z^d_+.$	
	\end{proof}
With Lemma~\ref{induction-lemma} in hand, let us prove Theorem~\ref{main-thm-new}.

    \begin{proof}[{\bf Proof of Theorem~\ref{main-thm-new}}]
For $h\in\mathcal H$, consider the function $\phi_{({\bf T},h)}: \mathbb{Z}_+^d \to \mathbb{R}$ given by
		$$\phi_{({\bf T},h)}(\alpha) := \|\mathbf{T}^\alpha h\|^2, \quad \alpha \in \mathbb{Z}_+^d.$$
Since $\mathbf T$ is completely hyperexpansive,
$\phi_{(\mathbf T,h)}$ is completely alternating. Its $i$-th coordinate
function is
\begin{equation*}
		\phi_{({\bf T},h)}^{(i)}(\alpha_i)=\|T_i^{\alpha_i} h\|^2\overset{\eqref{eq_onevarible}}{=}\|h\|^2+b_{(T_i,h)}(\alpha_i) +\int_{[0,1)}(1-x^{\alpha_i})d\nu_{(T_i,h)}(x).
		\end{equation*}
		Thus $\nu_{(T_i,h)}$ becomes the L\'evy measure for $\phi_{({\bf T},h)}^{(i)}$ for $i=1,2,\ldots,d.$
By Lemma~\ref{induction-lemma}, the identities
$D(T_i,T_j)=0$ for $i\neq j$ are equivalent to
\[
 \phi_{(\mathbf T,h)}(\alpha)
 =\sum_{i=1}^d
 \phi_{(\mathbf T,h)}^{(i)}(\alpha_i)
 -(d-1)\phi_{(\mathbf T,h)}(0)
\]
for every $\alpha\in\mathbb Z_+^d$ and every $h\in\mathcal H$.
The conclusion now follows directly from
Theorem~\ref{main-theorem}.
\end{proof}

	\section{Proof of Theorem~\ref{d-mu-1-2-decomp}}\label{dirichlet-space}
    Following \cite{A1975, M1988} (cf.~\cite{D1951}), for a finite positive Borel measure $\rho$ on the closed unit disc $\overline{\mathbb D}$, we define the \emph{$(i,j)$-th moment} of $\rho$ by
	\begin{equation}\label{moment-def}
	\widehat{\rho}(i,j)=\int_{\overline{\mathbb D}}\overline z^{i} z^{j} d\rho(z).
	\end{equation}
Let $\mathcal O(\mathbb D)$ denote the space of holomorphic functions on $\mathbb D$. Let $H^2(\mathbb D)=
\{f\in\mathcal O(\mathbb D):
\|f\|^2=\sum_{n=0}^\infty|\widehat f(n)|^2<\infty\}$ be the Hardy space of the unit disc. For a finite positive Borel measure $\rho$ on $\overline{\mathbb D}$, Aleman \cite[Chapter~IV]{A1993} introduced the  following notion of {\it Dirichlet-type spaces} $\mathcal D(\rho)$ on $\mathbb D:$
	\beqn
	\mathcal D(\rho):=\left\{f\in H^2(\mathbb D):\int_{\mathbb D}|f'(z)|^2U_{\rho}(z)\;dA(z)<\infty \right\},
	\eeqn
	where $U_{\rho}(z)=\int_{\mathbb D}\log \left |\frac{1-\overline{\zeta}z}{z-\zeta}\right |^2\frac{d\rho(\zeta)}{1-|\zeta|^2}+\int_{\mathbb T}\frac{1-|z|^2}{|\zeta-z|^2}d\rho(\zeta),$   $z\in \mathbb D$ is a superharmonic weight function, and dA is the normalized Lebesgue area measure on $\mathbb D$. The space $\mathcal D(\rho)$ is equipped with the norm
	\begin{equation*}
	\|f\|^2_{\mathcal D(\rho)}:=\|f\|^2_{H^2(\mathbb D)}+\int_{\mathbb D}|f'(z)|^2U_{\rho}(z)\;dA(z),\quad f\in \mathcal D(\rho).
	\end{equation*}
	These spaces have been extensively studied in \cite{A1993, EKKMR2016, GNS2018} (see also \cite{EKMR2014, R1991, RS1991}). In particular, in \cite[Theorem~1.10(i), p~76]{A1993}, it is shown that the multiplication operator $\mathscr M_z$ by the coordinate function $z$ on $\mathcal D(\rho)$ is a cyclic analytic completely hyperexpansive operator. Moreover, any cyclic analytic completely hyperexpansive operator is unitarily equivalent to $\mathscr M_z$ on $\mathcal D(\rho)$  for some finite positive Borel measure $\rho$ supported on $\overline{\mathbb D}$  (see \cite[Theorem~2.5, p.~79]{A1993}). Below we observe a relation between $\rho$ and the L\'evy measure for the pair $(\mathscr M_z, 1).$	
	\begin{lemma}\label{inner-prod}
		Let $\rho$ be a finite positive Borel measure on $\overline{\mathbb D}.$ For each $n\geq 0,$ define $\phi_{(\mathscr M_z, 1)} (n):= \|\mathscr M_z^n1\|^2_{\mathcal D(\rho)}.$ If $\mu_{(\mathscr M_z,1)}$ denotes the representing measure of
$\Delta\phi_{(\mathscr M_z,1)}$, and $S(z)=|z|^2$ for $z\in \overline{\mathbb D},$  then $\mu_{(\mathscr M_z, 1)}=S_{*}\rho.$
	\end{lemma}
	\begin{proof}
		We first show that 
		\beq\label{mu-hat-rho-hat}
		\widehat{\rho}(n,n)=\widehat{\mu}_{(\mathscr M_z, 1)}(n), \quad n\in \mathbb Z_+,
		\eeq
		where $\widehat{\mu}_{(\mathscr M_z, 1)}(n)= \int_{[0,1]}x^n d\mu_{(\mathscr M_z, 1)}(x)$.
		By \cite[Proposition 3.1]{B2025}, the norm of the monomials in $\mathcal{D}(\rho)$ is given by
		\beq\label{norm-D-sigma}
		\|z^n\|^2_{\mathcal D(\rho)}&=&1+\sum_{k=0}^{n-1}\widehat{\rho}(k,k),\quad n\geq 1.
		\eeq
On the other hand, 
		\beq\label{mu-hat}
		\phi_{(\mathscr M_z, 1)} (n)&=& 1+n\mu_{(\mathscr M_z, 1)}(\{1\})+ \int_{[0,1)} (1-x^n)\frac{d\mu_{(\mathscr M_z, 1)}(x)}{1-x}\notag \\
		\|\mathscr M_z^n1\|^2_{\mathcal D(\rho)}&=&1 +\sum_{k=0}^{n-1}\int_{[0,1]}x^{k}d\mu_{(\mathscr M_z, 1)}(x)\notag \\
		\|z^n\|^2_{\mathcal D(\rho)}&=&1 +\sum_{k=0}^{n-1}\widehat{\mu}_{(\mathscr M_z, 1)}(k),\qquad n\geq 0.
		\eeq
		Combining \eqref{norm-D-sigma}and \eqref{mu-hat} yields
		\beqn
		\sum_{k=0}^{n-1}\widehat{\rho}(k,k)=\sum_{k=0}^{n-1} \widehat{\mu}_{(\mathscr M_z, 1)}(k),\quad \text{ for all } n\geq 1.
		\eeqn
		Since the above equation is true for each $n\geq 1,$ using backward substitution we get \eqref{mu-hat-rho-hat}. 
		For $n\geq 0,$ consider $g:[0,1]\rightarrow \mathbb R$ as $g(t)=t^n$ for $t\in [0,1].$ Then 
		\beqn
		\widehat{\mu}_{(\mathscr M_z, 1)}(n)&=&\widehat{\rho}(n,n)\overset{\eqref{moment-def}}{=}\int_{\overline{\mathbb D}}|z|^{2n} d\rho(z)=\int_{\overline{\mathbb D}}(g\circ S)(z)d\rho(z)\\
		&\overset{\eqref{push-forward-int}}{=}&\int_{[0,1]} g(t)d(S_{*}\rho)(t)= \int_{[0,1]} t^nd(S_{*}\rho)(t)=\widehat{S_{*}\rho}(n).
		\eeqn
		Since the two measures have the same moments, uniqueness in the
Hausdorff moment problem yields $\mu_{(\mathscr M_z, 1)}=S_{*}\rho$ on $[0,1].$
	\end{proof}
	\begin{lemma}\label{lebesgue-coro}
Let $\sigma$ be the normalized Lebesgue area measure on
$\overline{\mathbb D}$. Then $S_*\sigma$ is the Lebesgue measure $m$ on
$[0,1]$.
\end{lemma}
	\begin{proof} 
		Note that for each $k\geq 0$,
		\beqn
	 \widehat{\sigma}(k,k)= \int_{\overline{\mathbb D}} |z|^{2k} d\sigma(z)=\frac{1}{\pi}\int_{\theta=0}^{2\pi}\int_{r=0}^1 r^{2k+1} dr d\theta=\frac{1}{k+1}=\widehat{m}(k).
		\eeqn
		For $k\geq 0,$ consider $g:[0,1]\rightarrow \mathbb R$ as $g(t)=t^k$ for $t\in [0,1].$ Then
		\beqn
		\widehat{S_{*}\sigma}(k)=\int_{[0,1]}g(t) d(S_{*}\sigma)(t)\overset{\eqref{push-forward-int}}{=}\int_{\overline{\mathbb D}} (g\circ S)(z) d\sigma(z)=\int_{\overline{\mathbb D}} |z|^{2k} d\sigma(z)=\widehat{\sigma}(k,k).
		\eeqn
		Finally, combining the above equations and using the uniqueness of the Hausdorff moment sequence yields $S_{*}\sigma=m.$
	\end{proof}

	Now we are ready to prove the main theorem.
	\begin{proof}[{\bf Proof of Theorem~\ref{d-mu-1-2-decomp}}]
		Note that the multiplication pair $\mathscr M_{z}=(\mathscr M_{z_1},\mathscr M_{z_2})$ on $\mathcal D({\rho^{(1)},\rho^{(2)}})$ is a completely hyperexpansive pair with defect zero (see \cite[Lemma~3.10]{B2025}). 
        Thus, by Theorem~\ref{main-thm-new}, the L\'evy measure for $(\mathscr M_{z}, 1)$ has the following decomposition
        \beq\label{nu-mz-nu-mz-i}
        \nu_{(\mathscr M_{z}, 1)}=\nu_{(\mathscr M_{z_1}, 1)}\times \delta_1+ \delta_1\times \nu_{(\mathscr M_{z_2}, 1)}.\eeq
        By \cite[Proposition~3.2]{B2025},
\[
 \|z_i^n\|_{\mathcal D(\rho^{(1)},\rho^{(2)})}^2
 =\|z^n\|_{\mathcal D(\rho^{(i)})}^2,
 \qquad n\in\mathbb Z_+.
\]
Hence Lemma~\ref{inner-prod} gives
$\mu_{(\mathscr M_{z_i},1)}=S_*\rho^{(i)},
 \qquad i=1,2.$
Using the one-variable relation between the representing and
L\'evy measures (see one-variable version of \eqref{mu-i-nu}), we obtain
\[
 d\nu_{(\mathscr M_{z_i},1)}(t)
 =\frac{1}{1-t}\,d(S_*\rho^{(i)})(t),
 \qquad t\in[0,1).
\]
Substitution into \eqref{nu-mz-nu-mz-i} now yields
\eqref{levy-k-diri}.
	\end{proof}

	The following is an immediate consequence of the above result.
	\begin{corollary}
		Let $\sigma$ be the normalized Lebesgue measure on the closed unit disc $\overline{\mathbb D}.$ Then the L\'evy measure $\nu_{(\mathscr M_z,1)}$ for the tuple $((\mathscr M_{z_1},\mathscr M_{z_2}),1)$ on $\mathcal D(\sigma, \sigma)$ is given by
		\begin{equation*}
		d\nu_{(\mathscr M_z,1)}(x)=\frac{1}{1-x_1}d(m\times\delta_1)(x)+\frac{1}{1-x_2}d(\delta_1\times m)(x) 
		\end{equation*}	
		on $[0,1]^2\backslash\{(1,1)\},$ where $m$ is the normalized Lebesgue measure on $[0,1].$
	\end{corollary}
	\begin{proof}
		Using Theorem~\ref{d-mu-1-2-decomp}, we get that 
		\begin{equation*}
		d\nu_{(\mathscr M_z,1)}(x)=\frac{1}{1-x_1}d((S_{*}\sigma)\times\delta_1)(x)+\frac{1}{1-x_2}d(\delta_1\times (S_{*}\sigma))(x).
		\end{equation*}	
		The rest of the proof follows from Lemma~\ref{lebesgue-coro}.
	\end{proof}		

\subsection*{Funding statement} The research of the first named author is supported by the Fellowship for Academic and Research Excellence (FARE ID: FA2508001) and the post-doctoral fellowship provided by IIT Kharagpur (ID: PDF2026176). The research of the second named author is supported by a post-doctoral fellowship provided by the National Board for Higher Mathematics (NBHM), India (Order No: 0204/21(22)/2025-R\&D- II/16318 dated December 09, 2025).

\end{document}